\documentclass[12pt]{article}

\usepackage[all]{xypic}
\usepackage{amsmath}
\usepackage{amsfonts}
\usepackage{amssymb}
\usepackage{amsthm}

\usepackage[top=2cm, bottom=2cm, left=2cm, right=2cm]{geometry}
\newtheorem{thm}{Theorem}[section]

\newtheorem{defn}[thm]{Definition}

\newtheorem{prp}[thm]{Proposition}

\newtheorem{lmm}[thm]{Lemma}

\newcommand {\mb}{\mathbb}
\newcommand {\Z}{\mb Z}
\newcommand {\R}{\mb R}

\newcommand {\F}{\mb F}

\newcommand {\colim}{\textrm{colim}\ }

\begin{document}

\title{On the image of the unstable Boardman map}

\author{Hadi Zare\\
        School of Mathematics, Statistics,
        and Computer Sciences\\
        University of Tehran, Tehran, Iran\\
        \textit{email:hadi.zare} at \textit{ut.ac.ir}\\
        and\\
        School of Mathematics, Institute for Research in Fundamental Sciences (IPM),\\
        P. O. Box. 19395-5746, Tehran, Iran}
\date{}

\maketitle

\begin{abstract}
We consider the `unstable Boardman map' (homomorphism if $k>0$)
$$b:\pi^{m+k}\Sigma^k\Omega^lS^{n+l}\simeq[\Omega^lS^{n+l},\Omega^kS^{m+k}]\longrightarrow \mathrm{Hom}(H_*\Omega^lS^{n+l},H_*\Omega^kS^{m+k})$$
defined by $h(f)=f_*$. We work at the prime $2$, with $k=0$, and determine the image for various in the following cases : (1) $m=n$ and $l>0$ arbitrary; (2) $m>n$ and $l=1$. We observe that in most of the cases the image is trivial with the exceptions corresponding to the cases when either there is a (commutative) $H$-space structure on $S^n$ or there is a Hopf invariant one element. 
\end{abstract}


\tableofcontents

\section{Introduction}
Suppose $E$ is a nice ring spectrum with the identity. Then there are two important homomorphisms which allow to study stable behaviour of a space $X$. The first one is a the stable Hurewicz homomorphism
$$h^s_E:\pi_n^sX\to E_nX$$
which sends $f$ to $(E_*f)x_n$ where $x^E_n\in E_n(S^n)$ is a generator provided by the unit $S^0\to E$. This homomorphism has been studies extensively for many important spectra such as $E=H\F_p,MU,BP,K,KO,tmf$ as there exists complete description of the stable Hurewicz homomorphism $h^s_E:{\pi_*^s}S^0\to E_*S^0$ for these spectra \cite{Adams-stablehomotopyandgeneralisedhomology}, \cite{Ravenel-Greenbook}, \cite{HopkinsMahowald}, \cite{Hopkins-ICM2002}. In fact computing the image of this homomorphism is so important that one tends to find a spectrum $E$ so that $h^s_E:{\pi_*^s}S^0\to E_*S^0$ is much closer to becoming an isomorphism and detects more and more elements in $\pi_*^sS^0$. From this point of view, there is an attempt to find bounds on the dimension/exponents of kernel and cokernel of this homomorphism (see \cite{Arlettaz-thestableHurewiczhomomorphism}, \cite{Mathew-exponents}).

The second homomorphism is the stable Boardman homomorphism
$$b^s_E:\pi^n_sX\to E^nX$$
which is defined by $b^s(f)=(E^*f)x^n_E$ where $x^n_E\in E^nX$ is a generator provided by the unit $S^0\to E$. There also has been detailed study on this homomorphism \cite{Hunton-Boardmanhomomorphism}, \cite{Arlettaz-Boardmanhomomorphism}.\\

These two homomorphisms are dual, possibly up to some degree shift, over $\pi_*E$ in a suitable sense (see \cite[Chapter 13]{Switzer} for a detailed discussion). Moreover, up to our knowledge, despite existence of some explicit relation among $h^s_E$ and $b^s_E$, there is no dictionary of the results about $h^s_E$ and $b^s_E$. We also note that this duality is not one that is induced by means of $S$-duality in the stable homotopy category in the sense that for a given $f:S^n\to X$ then $D(f)$ is not necessarily, up to finite number of suspensions, a map $X\to S^n$. However, knowing that these homomorphisms are dual in a suitable sense and the philosophy that sometimes solving a `dual' problem could be easier temps one to study one of these homomorphisms in order to justify some of claims about the other one. This duality could be used to obtain some information on the algebra. For instance, one may try to relate the rank of the image of $h^s_E$ to the rank of the kernel of $b^s_E$, etc.\\

The aim of this work is to apply and explore this idea in the presence of destabilisation functor $\Omega^\infty\Sigma^\infty$ when $E=H\Z/p$ for some prime $p$ with an special interest in the case of $p=2$. But, before proceeding further, there is one more piece of ingredient to bring in. We note that although given $f\in{\pi_n^s}X$ or $f\in{\pi^n_s}X$, knowing the specific images $h^s(f)$ or $b^s(f)$ could be important and useful, however, the primary aim of studying these homomorphisms is to understand and describe parts of ${\pi_n^s}X$ or ${\pi^n_s}X$ which could be detected using homology or cohomology. From this point of view, we may consider $h^s$ and $b^s$ as homomorphisms specific examples of more generalised Hurewicz, respectively Boardman homomorphisms
$$h^s:{\pi_n^s}X\longrightarrow \mathrm{Hom}(H_nS^n,H_nX),\ b^s:{\pi^n_s}X\longrightarrow\mathrm{Hom}(H^nS^n,H^nX).$$
Now, we describe the unstable Hurewicz and Boardman homomorphisms. Recall that for a spaces $X$ with base point, we have isomorphisms of groups
$$\pi_n^sX\simeq\pi_nQX\simeq [S^n,QX],\  \pi^n_sX\simeq [X,QS^n]$$
provided by the adjointness between $\Omega$ and $\Sigma$ functors where $QX=\colim\ \Omega^i\Sigma^i X$. This allows to define unstable Hurewicz and Boardman homomorphisms
$$h:{\pi_n^s}X\longrightarrow \mathrm{Hom}(H_nS^n,H_nQX),\ b:{\pi^n_s}X\longrightarrow\mathrm{Hom}(H^nQS^n,H^nX)$$
defined in the natural way $h(f)=f_*$ and $b(f)=f^*$. The evaluation map $\Sigma^\infty QX\to \Sigma^\infty X$ induces the stable homology suspension $\sigma_*^\infty:H_*QX\to H_*X$ and the stable cohomology suspension $\sigma^*_\infty:H^*X\to H^*QX$ which fit into commutative diagrams as
$$\xymatrix{
{\pi_n^s}X\ar[r]^-h\ar[rd]_-{h^s} & \mathrm{Hom}(H_nS^n,H_nQX)\ar[d]^-{\sigma_*^\infty} & {\pi^n_s}X\ar[r]^-{b^s}\ar[rd]_-{b} & \mathrm{Hom}(H^nX,H^nS^n)\ar[d]^-{\sigma^*_\infty}\\
                                                & \mathrm{Hom}(H_nS^n,H_nX)&& \mathrm{Hom}(H^nQX,H^nS^n)}$$

Next, notice that we have inclusion maps $\Omega^i\Sigma^iX\to QX$ which in the case of $h$ provides commutative diagram as
$$\xymatrix{
\pi_n\Omega^i\Sigma^i X\ar[r]^-h \ar[d] & \mathrm{Hom}(H_nS^n,H_n\Omega^i\Sigma^iX)\ar[d]\\
\pi_n^sX\simeq\pi_nQX\ar[r]^-h                        & \mathrm{Hom}(H_nS^n,H_nQX).}
$$
Also note that working at a prime $p$, the duality $H^nX\simeq \mathrm{Hom}_{\Z/p}(H_nX,\Z/p)$ between homology and cohomology provided by the Universal Coefficient Theorem, allows to consider $b^s$ and $b$ as homomorphisms
$$b^s:{\pi^n_s}X\longrightarrow\mathrm{Hom}(H_nX,H_nS^n),\ b:{\pi^n_s}X\longrightarrow\mathrm{Hom}(H_nX,H_nQS^n)$$
which send $f$ to $f_*$. For $b$, this provides a commutative diagram as
$$\xymatrix{
\pi^{n+i}\Sigma^iX\ar[r]^-{\simeq}&[X,\Omega^i\Sigma^iS^n]\ar[r]^-b\ar[d] & \mathrm{Hom}(H^n\Omega^iS^{n+i},H^nX)\ar[r]^-{\simeq} & \mathrm{Hom}(H_nX,H_n\Omega^iS^{n+i})\ar[d]\\
\pi^n_sX\ar[r]^-{\simeq}          &[X,QS^n]\ar[r]^-b          &  \mathrm{Hom}(H^nQS^n,H^nX)\ar[r]^-{\simeq}\ar[u] & \mathrm{Hom}(H_nX,H_nQS^n)
}$$
These observations motivates one to study the unstable Boardman homomorphisms
$$b:[X,\Omega^iS^{n+i}]\longrightarrow \mathrm{Hom}(H_nX,H_n\Omega^iS^{n+i})$$
where $i\geqslant 0$.
%\begin{rmk}
%By choosing $X=S^k$ the above unstable Boardman homomorphism $b:[X,\Omega^iS^{n+i}]\longrightarrow \mathrm{Hom}(H_nX,H_n\Omega^iS^{n+i})$ coincides with the unstable Hurewicz homomorphism
%$$\pi_k\Omega^iS^{n+i}\longrightarrow\mathrm{Hom}(H_nS^k,H_n\Omega^iS^{n+i}).$$
%\end{rmk}

\section{Statement of results}
The main motivation for this work has been the study of spherical classes in $H_*QS^n$ and $H_*\Omega^lS^{n+l}$ which is the same as the determination of Hurewicz homomorphisms
$$h:\pi_*\Omega^lS^{n+l}\to H_*\Omega^lS^{n+l},\ h:\pi_*QS^n\to H_*QS^n$$
where $n\geqslant 0$ and $l\leqslant 0$.  The problem of determining spherical classes in $H_*X$ is not always an easy problem, e.g. in the case of $X=QS^0=\colim \Omega^iS^i$ it is an open problem (see for example \cite{Curtis}, \cite{Wellington}, \cite{Za}). The problem of determining spherical classes in finite loop spaces $\Omega^lS^{n+l}$ also is an open, although some progress for small values of $l$ has been made where we have achieved complete classification of these classes (see \cite{Zare-Els-1} and \cite{Zare-Mexicana}). For this purpose, and following the above philosophy, we are interested in looking at the dual problem and studying the image of Boardman homomorphisms
$$b:[\Omega^lS^{n+l},\Omega^kS^{m+k}]\to \mathrm{Hom}(H_*\Omega^lS^{n+l},H_*\Omega^kS^{m+k})$$
where $k\leqslant \infty$ with the convention $\Omega^{\infty}\Sigma^{\infty}X=QX$. This motivates the following definition.

\begin{defn}
A generalised cospherical class in $H_*X$ is determined by a map $f:X\to \Omega^kS^{m+k}$, $k\geqslant 0$, so that $f_*\neq 0$. A reduced generalised cospherical class is determined by a map $f:X\to\Omega^kS^{m+k}$, $k\geqslant 0$ so that $f_*:H_iX\to H_i\Omega^kS^{m+k}$ is nontrivial for some $i>m$.
\end{defn}

Here, we are speaking loosely as preimage of $f_*$ could contain more than one element. If one is hesitated by our definition, then the may call any element of $f_*^{-1}(H_*\Omega^kS^{m+k}-\{0\})$ a generalised cospherical class; this is the point of view taken in \cite{Wu-Memoire2003} when studying cospherical classes $X\to S^m$ corresponding to the case $k=0$ in our definition. Indeed, one may decide to define $z\in H^*X$ to be a cospherical class if for some $f:X\to S^m$ we have $f^*(x_m)=z$ which is the point of view considered in \cite{LamDuane} where the authors consider maps $X\to S^m$ which induce nontrivial maps in $KO$-theory; hence they have studied c-spherical classes in $KO(X)$ for specific choices of $X$.\\

Finally, note that for $k>0$ the source of the unstable Boardman map
$$b:[\Omega^lS^{n+l},\Omega^kS^{m+k}]\to \mathrm{Hom}(H_*\Omega^lS^{n+l},H_*\Omega^kS^{m+k})$$
is the `unstable' group $[\Omega^lS^{n+l},\Omega^kS^{m+k}]$ whose complete computation needs a suitable unstable Adams spectral sequence (ASS). However, we do not attempt working with any unstable ASS and instead we try to use available geometric techniques to study the image of this homomorphism. We shall state our results integrally when possible. Otherwise, we often localise at $p=2$. The following is our first result.

\begin{thm}\label{maintheorem}
Suppose all spaces are localised at the prime $2$. Suppose $k=0$ the following statement hold.\\
(i) If $m=n=0$ then for any $l>0$ the image of $b$ is isomorphic to $\Z/2\{\iota\}$ where $\iota\in H^1QS^1$ is the fundamental class.\\
(ii) If $m=n>0$ and $l=1$ then the image of $b$ is nontrivial if and only if $n\in\{1,3,7\}$ and in this case the image is isomorphic to $\Z/2\{\partial\}$ where $\partial:\Omega S^{n+1}\to S^n$ is the boundary map in the Barratt-Puppe sequence for one of the Hopf maps $\eta,\nu,\sigma$.\\
(iii) If $m=n>1$ and $l>1$ then the image of $b$ is trivial.\\
(iv) If $m=n=1$ and any $l>1$ then the image of $b$ is isomorphic to $\Z/2\{\theta_{S^1}\}$ where $\theta_{S^1}:QS^1\to S^1$ corresponds to the structure map of $S^1$ as an infinite loop space.
\end{thm}

Next we consider the cases with $f:\Omega^lS^{n+1}\to S^m$. In this case, we have only results for $l=1$. Note that by James splitting if $f:\Omega S^{n+1}\to S^m$ satisfies $f_*\neq 0$ then $m=tn$ for some $t$. We have the following.

\begin{thm}\label{maintheorem2}
Suppose $f:\Omega S^{n+1}\to S^m$ with $m=tn>n$.
(i) For $m>n$ with $m=tn$ if $n+1$ is odd or both of $n+1$ and $t$ are even then the image of $b$ is trivial.\\
(ii) For $m>n$ with $m=tn$ if $n+1$ is even, $t$ is odd, and $tn\not\in\{3,7\}$ then the image of $b$ is trivial.\\
(iii) For $tn\in\{3,7\}$ there exists a map $f:\Omega S^2\to S^t$ with $H_t(f)\neq 0$.
\end{thm}

We have excluded the case with $tn=1$ as it implies that $t=n=1$. In this case, it is impossible to have $1=m>n=1$. In this case, we have $f:\Omega S^{1+1}\to S^1$ is one of the cases studies by the previous theorem.\\

\textbf{Acknowledgements.} Some of the results in this paper were presented in Dalian workshop on algebraic topology, 2018. I am grateful to the organisers, specially Fengchun Lei, and Jie Wu for the invitation and the hospitality. I am also grateful to Mark Grant for some communications on Theorem 4.5 which resulted in a corrected version of the proof. This research was in part supported by a grant from IPM (No.98470122).

\section{Proof of Theorem \ref{maintheorem}: Case of $m=n>0$}
\subsection{Case of $l=1$}
For $l=1$ there are examples at hand which are provided by Hopf fibrations, namely maps $\Omega S^{n+1}\to S^n$ for $n=1,3,7$. The existence of
these maps also provides a decomposition
$$\Omega S^{n+1}\simeq S^n\times\Omega S^{2n+1}.$$
We show these are the only possible cases (at least modulo $2$). We have the following formulation of Adams' Hopf invariant one element result.

\begin{lmm}\label{AdamsHopf}
The followings are equivalent.\\
(i) There is a map $f:\Omega S^{n+1}\to S^n$ with $f_*\neq 0$ (modulo $2$).\\
(ii) There is a map $g:S^{2n}\to\Omega S^{n+1}$ with $g_*\neq 0$ (modulo $2$).\\
(iii) There is a map $h:S^{2n+1}\to S^{n+1}$ of unstable Hopf invariant one, and $n\in\{1,3,7\}$.\\
(iv) There is a stable splitting $P^n\simeq S^n\vee P^{n-1}$ where $P^n$ is the $n$-dimensional real projective space.
\end{lmm}

\begin{proof}
$(i)\Rightarrow (ii):$ Let $i:S^n\to \Omega S^{n+1}$ be the inclusion adjoint the to identity $S^{n+1}\to S^{n+1}$. Then $(f\circ\iota)_*\neq 0$, hence
(at $p=2$) $f\circ \iota$ is homotopic to the identity. Together with James fibration $S^n\to\Omega S^{n+1}\stackrel{H}{\to}\Omega S^{2n+1}$ it follows
that
$$(f,H):\Omega S^{n+1}\to S^n\times\Omega S^{2n+1}$$
is a homotopy equivalence. The inclusion $S^{2n}\to \Omega S^{2n+1}$ gives rise to a spherical class. Consequently, the composition
$$g:S^{2n}\to\Omega S^{2n+1}\to\Omega S^{n+1}$$
gives rise to a spherical class in $H_*(\Omega S^{2n+1};\Z/2)$, so $g_*\neq 0$.\\
$(ii)\Rightarrow (iii):$ If $g_*\neq 0$ then from James' description of $H_*\Omega S^{2n+1}$ we see that $g_*(x_{2n})=x_n^2$. It is well known that the
adjoint of $g$, say $h:S^{2n+1}\to S^{n+1}$ has unstable Hopf invariant one.\\
$(iii)\Rightarrow (i):$ As noted above, (ii) and (iii) are equivalent. It follows that the adjoint of $h$, say $g:S^{2n}\to\Omega S^{2n+1}$ maps
nontrivially under $H_\#:\pi_{2n}\Omega S^{n+1}\to\pi_{2n}\Omega S^{2n+1}$. This implies that $H\circ g$ is homotopy the identity. The claimed
decomposition follows immediately.\\
$(ii)\Longrightarrow (iv)$: First recall that the projection on the top cell $P^n\to S^n$ is an unstable map which is nontrivial in $\Z/2$-homology. For $g:S^{2n}\to\Omega S^{n+1}$ being nontrivial in homology, consider the composition
$S^{2n}\to \Omega S^{n+1}\to QS^n$ and write $g$ for this map as well. By Kahn-Priddy theorem, $g$ lifts to $Q\Sigma^n P$, and yields a map $S^{2n}\to Q\Sigma^n P$ which is nontrivial in homology. By taking adjoint, we obtain a map $\widetilde{g}:S^n\to QP$ which is nontrivial in homology. For dimensional reasons, this implies that $h(\widetilde{g})=a_n$ where $h$ is the Hurewicz homomorphism $\pi_*QP\to H_*QP$. Moreover, by cellular approximation, we may restrict to a map $\widetilde{g}:S^n\to QP^n$ which satisfies $h(\widetilde{g})=a_n$. This implies that for the stable adjoint of this map as $S^n\to P^n$ is nontrivial in homology. Therefore the composition $S^n\to P^n\to S^n$ is nontrivial in $\Z/2$-homology. Noting that the cofibre of $P^n\to S^n$ is $\Sigma P^{n-1}$ gives the other stable piece, so $P^n\simeq S^n\vee P^{n-1}$.\\
$(iv)\Longrightarrow (ii)$: Choose a splitting map $S^n\to P^{n-1}\vee S^n\simeq P^n$ given by the inclusion into the second summand. Then the adjoint $S^n\to QP^n$ is nontrivial in homology, so the composition $S^n\to QP$. For dimensional reasons, the $n$-adjoint of this map $S^{2n}\to Q\Sigma^n P$ is nontrivial in homology. After composition with the Kahn-Priddy map $Q\Sigma^n P\to QS^n$ we obtain a map $S^{2n}\to QS^n$ which is nontrivial in homology. Also note that the inclusion $\Omega S^{n+1}\to QS^n$ is a $2n$ equivalence, so we obtain a map $S^{2n}\to \Omega S^{n+1}$ which is nontrivial in homology. This completes the proof.
\end{proof}

This settles down the case with $l=1$. In this case, the case of $l>1$ reduces to the case of $l=1$ in the following sense.

\subsection{Case of $l>1$}
Unlike the case of $l=1$, for $l>1$ the existence of maps $f:\Omega^l S^{n+l}\to S^n$ with $f_*\neq 0$ is not so immediate. For $n=1$ we may choose $f:\Omega^lS^{n+l}\to S^1$ to be any representative of the identity element of $H^1(\Omega^lS^{n+l};\Z)\simeq\Z$. It is immediate that $f$ is nonzero in $H_*(-;k)$ for $k=\Z,\Z/2$. There is another way to see existence of such
maps. Since $S^1$ is an infinite loop space, let $\theta:QS^1\to S^1$ be the structure map which has the property that the composition $S^1\to QS^1\to
S^1$ is identity. In particular, $\theta_*\neq 0$. Now, the composition
$$f:\Omega^lS^{l+1}\to QS^1\to S^1$$
satisfies $f_*\neq 0$.\\

For the remaining cases, we have the following nonexistence result. Let's note that if $f:\Omega^lS^{n+l}\to S^n$ exists with $f_*\neq 0$ then the composition $\Omega S^{n+1}\to \Omega^lS^{n+l}\to S^n$ also would be nontrivial in $H_n(-;k)$, hence by Lemma \ref{AdamsHopf} we see that $n$ must be either $1$, $3$, or $7$. As we considered the case of $n=1$ above then we only need to resolve the cases $n\in\{3,7\}$.

\begin{lmm}\label{nonexistnece1}
(i) Suppose $n=3$. Then, there exists no map $f:\Omega^2S^5\to S^3$ with $f_*\neq 0$. Consequently, for $2\leqslant l\leqslant +\infty$ there exists no
map $f:\Omega^lS^{l+3}\to S^3$ with $f_*\neq0$.\\
(ii) Suppose $n=7$. Then, these exists no map $f:\Omega^2S^9\to S^7$ with $f_*\neq 0$. Consequently, for $2\leqslant l\leqslant +\infty$ there exists no
map $f:\Omega^lS^{l+7}\to S^3$ with $f_*\neq0$.\\
\end{lmm}

\begin{proof}
Proof of (i) and (ii) are similar. First note that the general case follows from our claim for double loop spaces as follows. For instance, note that
for $l\geqslant 2$, the composition
$\Omega^2S^5\to\Omega^lS^{l+3}$ is nonzero in $H_3(-;k)$ with $k=\Z,\Z/2$. Hence, existence of any map $f:\Omega^lS^{l+3}\to S^3$ with $f_*\neq 0$ would
imply that the composition $\Omega^2S^5\to\Omega^lS^{l+3}\to S^3$ is nonzero in homology, giving the desired contradiction.\\
Now we show there is no map $f:\Omega^2S^{n+2}\to S^{n}$ (with $n=3,7$) so that $f_*\neq 0$. We work at the prime $2$. Given a map
$f:\Omega^2\Sigma^2X\to X$ we may define $\mu:X\times X\to X$ by the following composition
$$\mu:X\times X \to *\times_{\Sigma_2}(X\times X)\to F(\R^2,2)\times_{\Sigma_2}(X\times X)\to\Omega^2\Sigma^2X\to X$$
where the first map on the left is projection, second and third maps are inclusion, and the last map is $f$. This map is a commutative multiplication on
$X$. Since the composition $S^n\to \Omega^2\Sigma^2S^n\to S^n$ is nonzero in homology (we may assume it is multiplication of degree $1$), hence it is homotpic to the identity. On the other hand, by construction, for a based path connected space $X$, the inclusion $X\to\Omega^2\Sigma^2 X$ can be decomposed as a composition
$$X\stackrel{\alpha}{\to}X\times X\to \Omega^2\Sigma^2 X$$
where $\alpha$ can be taken as either $(1,*)$ or $(*,1)$ with $*$ being the base point of $X$. This implies that, for $X=S^3,S^7$, $(X,\mu,*)$ is a commutative $H$ space in the sense of \cite{Gray}. But this is a contradiction as it is known that $S^3$ and $S^7$ do not admit any commutative
$H$-space structure.
\end{proof}

\section{Case of $m=n=0$}
\begin{lmm}
For any $l>0$ there exists a map $f:\Omega^lS^l\to S^0$ with $f_*\neq 0$. Moreover, any such $f$ we have $f=\Omega\iota$ where $\iota:\Omega^{l-1}S^l\to P=K(\Z/2,1)$ represents the fundamental class.
\end{lmm}

\begin{proof}
Let $l>0$, and take the fundamental class $\iota\in H^1QS^1\simeq\Z/2$ which can be realised as a map $\iota:QS^1\to P$. Define $f$ to be the composition
$$\Omega^lS^l\to QS^0\stackrel{\Omega\iota}{\to}\Z/2=S^0.$$
Clearly, $f_*\neq 0$. On the other hand, if $f:\Omega^lS^l\to S^0=\Omega P$ is given with $f_*\neq 0$, then the adjoint of $f$, say $\widetilde{f}:\Sigma\Omega^l S^l\to P$ is nontrivial in homology. Also, note we may consider the composition
$$\Sigma \Omega ^lS^l\stackrel{e}{\to}\Omega^{l-1}S^l\stackrel{\iota}{\to} P$$
which is nontrivial in homology where $e$ is the evaluation map (adjoint to the identity). Since $\widetilde{f},\iota\circ e\in [\Sigma\Omega^lS^l,P]\simeq\Z/2$ and both elements are nontrivial, therefore
$$\widetilde{f}=\iota\circ e\Rightarrow f=\Omega\iota\circ\widetilde{e}=\Omega\iota\circ 1=\Omega\iota.$$
This completes the proof.
\end{proof}

\section{Proof of Theorem \ref{maintheorem2}}

\subsection{The cases of $m=tn$ with $n+1$ odd, or $n+1$ and $t$ even}
We begin with the case $n=0$. In this case, as $\Omega S^1\simeq\Z$ it is not possible to find $m>0$ so that $g:\Omega S^1\to S^m$ is nontrivial in homology. Therefore, we consider the cases with $n>0$. If $f:\Omega S^{n+1}\to S^m$ is given with $f_*\neq 0$ then $(\Sigma f)_*\neq 0$. In the light of James splitting, $\Sigma\Omega S^{n+l}\simeq\bigvee_{t=1}S^{tn+1}$, $m=tn$ for some $t>1$ are the only possible choices for which $(\Sigma f)_*\neq0$ and consequently $f_*\neq 0$ can happen. That is

\begin{prp}
For $f:\Omega S^{n+1}\to S^m$ with $m\neq tn$ for all $t$, we have $f_*=0$.
\end{prp}

Our main result in this section is the following.

\begin{thm}\label{main-nonexistence1}
(i) Suppose $n+1$ is odd, $n>0$, and $t>1$. Then for any $f:\Omega S^{n+1}\to S^{tn}$ we have $f_*=0$ in $H_*(-;k)$ with $k=\Z,\Z/2$.\\
(ii) Suppose $n+1$, $n\geqslant1$, and $t>1$ are even. Then for any $f:\Omega S^{n+1}\to S^{tn}$ we have $f_*=0$ in $H_*(-;k)$ with $k=\Z,\Z/2$.\\
\end{thm}

\begin{proof}
The proof in both cases is by contradiction. By abuse of notation, we write $x_i$ for a generator of $H^iS^i$. Let's recall that
$$H^*(\Omega S^{n+1};\Z)\simeq\left\{\begin{array}{ll}
                                \Gamma_{\Z}(x_{n}) & \textrm{if }n+1\textrm{ is odd}\\
                                \Gamma_{\Z}(x_{2n})\otimes_{\Z}E_{\Z}(x_{n}) & \textrm{if }n+1\textrm{ is even}
                                \end{array}\right.$$
where $\Gamma_{\Z}$ and $E_{\Z}$ denote divided power algebra and exterior algebra functors over $\Z$ respectively. We first prove the integral result. \\
(i) Suppose $m=tn$ with $t>1$ and $f:\Omega S^{n+1}\to S^{tn}$ is given with $f_*\neq 0$ in $\Z$-homology. For $n>1$ with $n+1$ odd, and $t>1$, for dimensional reasons, for $X=\Omega S^{n+1},S^{tn}$ there is a duality between homology and cohomology with $H^*X\simeq\mathrm{Hom}_{\Z}(H_*X,\Z)$. It follows that $f^*\neq 0$. If $n+1$ is odd then having $f^*(x_{tn})\neq 0$ we see that $(f^*(x_{tn}))^2\neq 0$ in $\Gamma_{\Z}(x_{n})$ which implies that $x_{tn}^2\neq 0$ in $H^*S^{tn}$ which is an obvious contradiction. Hence, $f_*=0$.\\
(ii) First suppose $n>1$. Let's note that for $n>1$, $H_*\Omega S^{n+1}$ has nontrivial homology only in dimensions $kn$, with $k>0$, by James' splitting. So, if $H_i\Omega S^{n+1}\neq 0$ then $H_{i+\epsilon}\Omega S^{n+1}\simeq 0$ for $\epsilon\in\{-1,1\}$. This property is essential while appealing to the Universal Coefficient Theorem. Obviously, $H_*\Omega S^2$ does not have this property. We just note that choosing $t$ even forces $f^*(x_{tn})\in\Gamma_{\Z}(x_{2n})$. The rest of the proof goes exactly as part (i) and we leave the details to the reader.\\

Next, suppose either $n+1$ is odd with $n>0$ and $t>1$ or $n+1$ and $t$ both are even with $t,n>1$. Assume $f$ is given with $f_*\neq 0$ in $\Z/2$-homology. By the duality between homology and cohomology over $\Z/2$ we see that $f^*\neq 0$ in $\Z/2$-cohomology. For dimensional reasons, implied by our choices of $n$ and $t$, together with the naturality of the Universal Coefficient Theorem, we have a commutative diagram
$$\xymatrix{
0\ar[r] & \mathrm{Ext}(H_{tn-1}(S^{tn};\Z),\Z/2)\ar[r]\ar[d] & H^{tn}(S^{tn};\Z/2)\ar[r]\ar[d]^-{f^*} & \mathrm{Hom}(H_{tn}(S^{tn};\Z),\Z/2)\ar[r]\ar[d]^-{\mathrm{Hom}(f_*,\Z/2)} & 0\\
0\ar[r] & \mathrm{Ext}(H_{tn-1}(\Omega S^{n+1};\Z),\Z/2)\ar[r] & H^{tn}(\Omega S^{n+1};\Z/2)\ar[r] & \mathrm{Hom}(H_{tn}(\Omega S^{n+1};\Z),\Z/2)\ar[r] & 0}$$
The $\mathrm{Ext}$ terms do vanish for dimensional reasons, hence yielding a commutative diagram as
$$\xymatrix{
H^{tn}(S^{tn};\Z/2)\ar[r]^-{\simeq}\ar[d]^-{f^*} & \mathrm{Hom}(H_{tn}(S^{tn};\Z),\Z/2)\ar[d]^-{\mathrm{Hom}(f_*,\Z/2)}\\
H^{tn}(\Omega S^{n+1};\Z/2)\ar[r]^-{\simeq} & \mathrm{Hom}(H_{tn}(\Omega S^{n+1};\Z),\Z/2). }$$
We then conclude that the nonvanishing of $f^*$ in $\Z/2$-cohomology implies the nonvansihing of $\mathrm{Hom}(f_*,\Z/2)$ and consequently, nonvanishing of $f_*$ in $\Z$-homology. The result now follows by appealing to the integral case.
%Since $Sq^1x_{tn}=0$ hence $Sq^1f^*(x_{tn})=0$. By Bockstein exact sequence for cohomology associated to $\Z\stackrel{\times 2}{\to}\Z\to\Z/2$ together with its naturality with respect to mapping among spaces, we deduce that $f_*\neq 0$ in $\Z$-homology. The previous part now gives the desired contradiction. Hence, $f_*=0$.\\

Finally, $n=1$ and $t$ even, notice that by the Hopf invariant one result $\Omega S^2\simeq \Omega S^3\times S^1$ and the above decomposition of cohomology is becomes a result of this splitting. The fact that $t$ is even implies that the composition $\Omega S^3\stackrel{\Omega\eta}{\to}\Omega S^2\to S^{t}$ is nonzero in homology where $t$ is even. But, this leads to a contradiction by part (i).\\

\end{proof}

\subsection{Attaching by Whitehead products}
The content of this subsection is probably well known, but we include an exposition for further reference. The idea is that given a $CW$-complex which splits after finite number of suspensions, or stable splits, certain attaching maps are obtained from Whitehead product. The main of this subsection is to make such statements more precise. In the case of splitting after finite number of suspensions, it is enough to deal with the case of splitting after one suspension. During this section and afterwards, for a CW-complex $X$, write $X^k$ for its $k$-skeleton. We write $X_k=X/X^{k-1}$ and $X^m_n=X^m/X_{n-1}$. We have the following.

\begin{lmm}\label{Whiteheadattaching1}
Suppose $X$ is a $CW$-complex with finite number of cells in each dimension. Suppose $X$ has only one cell in dimensions $n$ and $2n$. Suppose $\Sigma X^{2n}\simeq S^{2n+1}\vee \Sigma X^{2n-1}$ then the attaching map of the $2n$-cell in $X$, resulting in the $S^{2n+1}$ summand of $\Sigma X^{2n}$, is ``obtained'' by a Whitehead product on the $n$-cell. More precisely, write $\alpha$ for the attaching map $S^{2n-1}\to X^{2n-1}$, $p:X^{2n-1}\to X^{2n-1}_n$ for the pinching map, and $i:S^n\to X_n^{2n-1}$ for the inclusion of the bottom cell. If $p\circ\alpha\neq 0$ then
$$p\circ \alpha=i_\#[\iota_n,\iota_n].$$
\end{lmm}

\begin{proof}
Suppose $\alpha:S^{2n-1}\to X^{2n-1}$ is the attaching map of the desired $2n$-cell. We may compose $\alpha$ with the projection $X^{2n-1}\to X_n^{2n-1}$. The fact that $X^{2n}=X^{2n-1}\cup_{\alpha}e^{2n}$ after one suspension splits as $\Sigma X^{2n}\simeq S^{2n+1}\vee \Sigma X^{2n-1}$ is the same as saying that $\alpha$ being to the kernel of the suspension map $E:\pi_{2n-1}X^{2n-1}\to\pi_{2n-1}\Omega\Sigma X^{2n-1}$. The splitting of $\Sigma X^{2n}$ also implies that $\Sigma X^{2n}_n\simeq \Sigma X^{2n-1}_n\vee S^{2n+1}$. The space $X_n^{2n-1}$ has its bottom cell in dimension $n$, so we may use the $\mathrm{EHP}$-sequence in the meta-stable range. In particular, consider the following portion of the $\mathrm{EHP}$-sequence
$$\pi_{2n-1}\Omega^2\Sigma(X_n^{2n-1}\wedge X_n^{2n-1})\stackrel{P}{\longrightarrow} \pi_{2n-1}X_n^{2n-1}\stackrel{E}{\longrightarrow}\pi_{2n-1}\Omega\Sigma X_n^{2n-1}.$$
Since $E(p\circ\alpha)=0$ then $p\circ\alpha=P\beta$ for some $\beta\in\pi_{2n-1}\Omega^2\Sigma(X_n^{2n-1}\wedge X_n^{2n-1})$. Since $X$ has only one cell in dimension $n$ then the space $\Sigma(X_n^{2n-1}\wedge X_n^{2n-1})$ has its bottom cell in dimension $2n+1$, by Hurewicz Theorem $\Omega^2\Sigma(i\wedge i)$ is an isomorphism and there exists a unique nonzero $\beta\in\pi_{2n-1}\Omega^2\Sigma(X_n^{2n-1}\wedge X_n^{2n-1})\simeq\pi_{2n-1}\Omega^2 S^{2n+1}$ with $p\circ\alpha=P\beta$. Writing $i:S^n\to X_n^{2n-1}$ for the inclusion we have the following commutative diagram
$$\xymatrix{
\pi_{2n-1}\Omega^2 S^{2n+1}\ar[r]^-P\ar[d]_-{(\Omega^2\Sigma(i\wedge i))_\#} & \pi_{2n-1}S^n\ar[r]^-E\ar[d]_-{i_\#} & \pi_{2n-1}\Omega\Sigma S^n\ar[d]^-{(\Omega\Sigma i)_\#}\\
\pi_{2n-1}\Omega^2\Sigma(X_n^{2n-1}\wedge X_n^{2n-1})\ar[r]^-P               & \pi_{2n-1}X_n^{2n-1}\ar[r]^-E        & \pi_{2n-1}\Omega\Sigma X_n^{2n-1}.
}$$
It follows that $p\circ\alpha=i_\#P\beta$. However, it is known that $P:\pi_{2n-1}\Omega^2 S^{2n+1}\to \pi_{2n-1}S^n$ is the Whitehead product. Since, we work at the prime $2$, it allows us to think of $P\beta$ as the Whitehead product $[\iota_n,\iota_n]$ which implies that
$$p\circ\alpha=i_\#[\iota_{n},\iota_{n}].$$
This completes the proof.
\end{proof}

\subsection{The case of $n+1$ even and $t$ odd}
We localise at the prime $2$.\\

\begin{thm}\label{elimination-Whitehead1}
Suppose $f:\Omega S^{n+1}\to S^{tn}$ where $t$ is odd and $n+1$ is even. If $tn\not\in\{1,3,7\}$ then $f_*=0$.
\end{thm}

\begin{proof}
The proof is by contradiction, and we wish to apply Lemma \ref{Whiteheadattaching1} to prove the Theorem. The space $X=\Omega S^{n+1}$ satisfies conditions of Lemma \ref{Whiteheadattaching1}. Suppose $tn\not\in\{1,3,7\}$ and $f_*\neq 0$. Let $i:S^{tn}\to X^{2tn-1}_{tn}$ denote the inclusion of the bottom cell, $\alpha:S^{2tn-1}\to X^{2tn-1}$ the attaching map of the $2tn$-cell, and $p:X^{2tn-1}\to X^{2tn-1}_{tn}$ the pinching map which collapses $X^{tn-1}$. Then, by Lemma \ref{Whiteheadattaching1} we have $$i_\#[\iota_{tn},\iota_{tn}]=p\circ\alpha.$$
Moreover, for dimensional reasons, the restriction of $f$ to the $(2tn-1)$-cell of $X$, say $f|_{2tn-1}:X^{2tn-1}\to S^{tn}$ extends to a map $g:X^{2tn-1}_{tn}\to S^{tn}$ so that $gp=f|_{2tn-1}$. These data yield a commutative diagram as following.
$$\xymatrix{S^{2tn-1}\ar[r]^-{\alpha}\ar[d]_-{[\iota_{tn},\iota_{tn}]}& X^{2tn-1}\ar[d]^-p\ar[r]^-{f|_{2tn-1}} & S^{tn} \\
            S^{tn}\ar[r]^-i                                    & X^{2tn-1}_{tn}\ar[ru]_-g.}$$
On the other hand $f|_{2tn-1}$ has to extend over the $2tn$-skeleton of $X$ to a map $f|_{2tn}:X^{2tn-1}_{tn}\to S^{tn}$ where the necessity being implied by the existence of $f:X\to S^{tn}$. The existence of $f|_{2tn}$ is the same as saying that $f|_{2tn-1}\circ \alpha=0$. On the other hand $i$ and $g$ are both nonzero in $H_{tn}(-;\Z/2)$ and working at the prime $2$ this implies that we may take $g\circ i$ is the identity (over $\Z$ this implies that $g\circ i$ is an odd multiple of the identity). It follows that
$$[\iota_{tn},\iota_{tn}]=g_\#\circ i_\#[\iota_{tn},\iota_{tn}]=f|_{2tn-1}\circ \alpha=0.$$
It is known that for $m$ odd the Whitehead product $[\iota_m,\iota_m]\in\pi_{2m-1}S^m$ vanishes if and only if there is a Hopf invariant one element in $\pi_{2m-1}S^m$
\cite[Corollary 1.3, Remark 1.4]{Cohen-acourse} which together with Adams Hopf invariant one result means that $m\in\{1,3,7\}$ \cite{Adams-Hopfinv}. This shows that if $f$ extends over the $2tn$-skeleton then $tn\in\{1,3,7\}$. This gives the desired contradiction. Hence, $f_*=0$.
\end{proof}

It remains to decide about the cases with $tn\in\{1,3,7\}$ with $t$ odd and $n+1$ even with $n>0$. However, note that we have required $t>1$, so $t\geqslant 3$. Also, $n$ has to be odd. This leave us with the cases where $n=1$ and $t=3,7$. That is we have to decide about the existence of maps $\Omega S^2\to S^t$ with $t\in\{3,7\}$ such that $f_*\neq 0$. We have the following existence result and we note that its proof is also valid for the case of $\Omega S^2\to S^1$.

\begin{thm}\label{existenceHopf1}
For $t\in\{1,3,7\}$ there exists a map $f:\Omega S^2\to S^t$ with $f_*\neq 0$.
\end{thm}

\begin{proof}
For the given values of $t$, the existence of Hopf invariant one elements in $\pi_t^s$ is equivalent to existence of a decomposition $\Omega S^{t+1}\simeq \Omega S^{2t+1}\times S^t$. The examples of Hopf invariant one elements, which exist integrally and before any localisation, is provided by the Hopf fibre bundle $S^t\to S^{2t+1}\to S^{t+1}$ which could be extended to a fibre sequence
$$\Omega S^{2t+1}\to \Omega S^{t+1}\stackrel{\partial}{\to} S^t\to S^{2t+1}\to S^{t+1}.$$
We outline the proof for the case of $t=3$ and the case of $t=1,7$ are similar. In this case consider following portion of the above sequence
$$\Omega S^7\stackrel{\Omega\nu}{\to}\Omega S^4\stackrel{\partial}{\to}S^3.$$
We choose the more functorial, although less usual, notation $H_k(g)$ to denote $H_k(g;a)$ which is $g_*$ on the $k$-th homology with coefficients in a PID $a$ with an interest in $a=\Z,\Z/2$. We construct a map $g:\Omega S^2\to \Omega S^4$ with the property that $H_3(g)\neq 0$ and $H_3(\partial g)\neq 0$. It then would suffice to define $f=\partial g$. \\

By James splitting $\Sigma\Omega S^2\simeq\bigvee_{k=1}^{+\infty} \Sigma S^k$ there is an isomorphism
$$[\Omega S^2,\Omega S^4]\simeq [\Sigma\Omega S^2,S^4]\simeq [\bigvee_{k=1}^{+\infty} \Sigma S^k,S^4]\simeq\bigoplus_{k=1}^{+\infty} [S^{k+1},S^4]\simeq\bigoplus_{k=1}^{+\infty}\pi_{k+1}S^4.$$
For $g:\Omega S^2\to \Omega S^4$, write $\mathrm{ad}(g):\Sigma \Omega S^k\to S^4$ for its adjoint and $g_k:S^{k+1}\to S^4$ for the restriction of $\mathrm{ad}(g)$ to the $\Sigma S^k$ summand of $\Sigma\Omega S^2$. Consider the commutative diagram
$$\xymatrix{
[\Omega S^2,\Omega S^4]\ar[r]^-{\simeq}\ar[d]_-{H_3(-)} & [\Sigma\Omega S^2,S^4]\ar[d]^-{H_4(-)}\\
H_3\Omega S^4\ar[r]^-{\simeq}                           & H_4S^4}
$$
by which we may write
$$H_3(g)=H_4(\mathrm{ad}(g))=\sum_{k=1}^{+\infty}H_4(g_k)=H_4(g_3)$$
as for dimensional reasons $H_4(g_k)=0$ if $k\neq 3$. Now, to construct $g$ with $H_3(g)\neq 0$, choose $g_3:S^{3+1}\to S^4$ to be any nonzero element $d\in\pi_{4}S^4\simeq\Z$. Choose $g_k$ for $k\neq 3$ to be any arbitrary element. It follows that $H_3(g)=d$. If we wish to provide a proof which is valid at any prime then choose $g_3$ to be the mapping corresponding to the identity $S^4\to S^4$ which yields $H_3(g)=1$. Note that $g$ is not in the image of
$$(\Omega \nu)_\#:[\Omega S^2,\Omega S^7]\to[\Omega S^2,\Omega S^4].$$
For if $g=(\Omega\nu)_\# g'$ for some $g':\Omega S^2\to\Omega S^7$ then $H_3(g')\neq 0$. However, this is impossible for dimensional reasons as $H_3(\Omega S^7)\simeq 0$. Therefore, $g$ does not lift to $\Omega S^7$ and $f=\partial g:\Omega S^2\to S^3$ is a nontrivial element. On the other hand the mapping $\partial$ is a $5$-equivalence whose inverse is provided by the inclusion $S^3\to\Omega S^4$. Therefore, $H_3(f)\neq 0$. This completes the proof.
\end{proof}

\section{Notes on the case with $l>1$ and $m>n$}
We can use some geometry to prove -

\begin{lmm}
For any $l>1$ and $f:\Omega^l S^{n+l}\to S^{2n}$ we have $f_*=0$.
\end{lmm}

\begin{proof}
The inclusion $i:\Omega S^{n+1}\to \Omega^lS^{n+l}$ is a $2n$-equivalence. Hence, assuming $f:\Omega^2S^{n+2}\to S^{2n}$ is nontrivial in homology, $g:\Omega S^{n+1}\stackrel{i}{\to}\Omega^lS^{n+l}\stackrel{f}{\to}S^{2n}$ satisfies $g_*\neq 0$. This contradicts Theorem \ref{main-nonexistence1}.
\end{proof}

We postpone further investigation on this to a future work.

%\bibliographystyle{plain}
%\bibliography{spherical}

\bibliographystyle{plain}

\end{document}